\documentclass[12pt,leqno,fleqn]{amsart}  
\usepackage{amsmath,amstext,amsthm,amssymb,amsxtra}
\usepackage[top=1.5in, bottom=1.5in, left=1.25in, right=1.25in]	{geometry}
\usepackage{txfonts} 
\usepackage[T1]{fontenc}
\usepackage{lmodern}

 \usepackage{euler}   

\allowdisplaybreaks

\usepackage{mathtools}
\mathtoolsset{showonlyrefs,showmanualtags}

\usepackage{hyperref} 
\hypersetup{
    colorlinks=true,       
    linkcolor=blue,          
    citecolor=magenta,        
    filecolor=magenta,      
    urlcolor=cyan           
}

\usepackage[msc-links]{amsrefs}

\theoremstyle{plain} 
\newtheorem{lemma}[equation]{Lemma} 
 
\newtheorem{theorem}[equation]{Theorem}

\theoremstyle{definition}
\newtheorem{definition}[equation]{Definition} 

\theoremstyle{remark}
\newtheorem{remark}[equation]{Remark}

\numberwithin{equation}{section}

\newcommand{\intav}{-\!\!\!\!\!\!\int}

\def\dint{\displaystyle\int}

\def\gs{\gtrsim}

\title[Compactness of commutator in the two weight setting] {Compactness of commutator of Riesz transforms\\ in the two weight setting}
\author{Michael Lacey and Ji Li}   

\address{ School of Mathematics, Georgia Institute of Technology, Atlanta GA 30332, USA}
\email {lacey@math.gatech.edu}
\thanks{The first author is a 2020 Simons Fellow. Research supported in part by grant  from the US National Science Foundation, DMS-1949206}

\address{ Department of Mathematics, Macquarie University, NSW 2109, Australia}
\email {ji.li@mq.edu.au}
\thanks{The second author is supported by ARC DP 170101060}

\begin{document}

\begin{abstract}
We characterize the compactness of commutators in the Bloom setting. Namely, for a suitably non-degenerate Calder\'on--Zygmund operator $ T$, 
 and a pair of weights $ \sigma , \omega  \in A_p$,  the commutator $ [T, b]$ is compact from $ L ^{p} (\sigma ) \to L ^{p} (\omega )$ 
 if and only if $ b \in VMO _{\nu }$, where $ \nu = (\sigma / \omega ) ^{1/p}$.   This extends the work of the first author, Holmes and Wick.  
 The weighted $VMO$ spaces  are different from the classical $ VMO$ space. In dimension $ d =1$,  compactly supported and smooth functions are dense in $ VMO _{\nu }$, but this need not hold in dimensions $ d \geq 2$.  Moreover, the commutator in the product setting with respect to little VMO space is also investigated. 
\end{abstract}

	\maketitle  

\section{Introduction} 

Let $\sigma $ be a weight on $\mathbb R^d$, i.e. a function that is positive almost everywhere and is locally integrable. For $1<p<\infty$, define $L^p(\sigma)$ to be the space of functions $f$ satisfying 
$\|f\|_{L^p(\sigma)}:= \big(\int_{\mathbb R^d} |f(x)|^p \sigma(x)dx\big)^{1/p}<\infty$.

In \cite{B}, Bloom considered the behavior of the commutator $[b,H]: L^p(\sigma)\mapsto L^p(\omega)$, where $H$ is the Hilbert transform on $\mathbb R$. When $\sigma=\omega=1$, it is well-known that the boundedness of $[b,H]$ is characterized by $b\in BMO(\mathbb R)$ \cite{CRW,N}, and the compactness of $[b,H]$ is characterized by $b\in VMO(\mathbb R)$ \cite{U}. Bloom worked out the setting for $\sigma, \omega\in A_p(\mathbb R)$ and $\sigma\not=\omega$. He showed that for $ 1< p < \infty $,
two weights $  \omega , \sigma  \in A_p$ and $ \nu = (\sigma /\omega ) ^{1/p}$, $[b,H]$ is bounded from $ L^p(\sigma)$ to $L^p(\omega)$ if and only if the symbol $b$ is in the weighted BMO space $BMO_\nu(\mathbb R)$ (for the definitions of $A_p$ weight and $BMO_\nu(\mathbb R)$, see section 2). Recently, Holmes, Wick and the first author \cite{HLW} established this characterization of two-weight boundedness for the commutator of Riesz transforms $[b,R_j]$ in $\mathbb R^d$, $d\geq2$, $j=1,\ldots,d$, using a new method via representation formula from Hyt\"onen and decomposition via paraproducts. It was further studied by Lerner--Ombrosi--Rivera-R\'ios \cite{LOR} via sparse domination, and by Hyt\"onen \cite{H2} via a new weak factorization technique (comparing to \cite{U1}).

However, characterization of the two-weight compactness of $[b,R_j]$, $j=1,\ldots,d$, ($[b,H]$ in dimension 1) is missing up to now. We fill in
this gap as follows.

\begin{theorem}\label{t:Main}  Suppose $ 1< p < \infty $,
two weights $  \lambda  , \sigma  \in A_p$ and $ \nu = (\sigma /\lambda  ) ^{1/p}$. Suppose that $ b\in BMO _{\nu } (\mathbb R^d )$, $R_j$ is the $j$-th Riesz transform on $\mathbb R^d$, $j=1,2,\ldots,d$. 
Then we obtain that $ b\in VMO _{\nu } (\mathbb R^d )$ if and only if the commutator
\ $[b,R_j]:  L ^{p} (\sigma )  \to  L ^{p} (\lambda  ) $\  is compact.
\end{theorem}

Some approaches to this Theorem will not succeed.  For instance,    Uchiyama \cite{U1}  used the fact that   $VMO(\mathbb R^d) $ is the closure of $C_0^\infty(\mathbb R^d)$ (smooth functions with compact support) under the $BMO(\mathbb R^d)$ norm. However, this is not necessarily true in the two weight setting when $d\geq2$, see \S\ref{s:VMO}.  
Another recent argument of Hyt\"onen  \cite{H3} shows that compactness can be extrapolated, recovering compactness of commutators in the 
one weight setting.  
It would be interesting to extend that argument to the two weight setting.

The sufficiency of Theorem \ref{t:Main} holds more generally. We note that this sufficiency argument holds not just for Riesz transforms but also for general Calder\'on--Zygmund operators $T$ with the associated kernel $K(x,y)$. We formulate it as follows.
Recall that a Calder\'on-Zygmund  operator $T$ (bounded on $L^2(\mathbb R^d)$) associated to a $\delta$-standard kernel $K(x, y)$ is an integral
operator defined initially on $f\in C_0^\infty(\mathbb R^d)$:
$$ T(f)(x):=\int_{\mathbb R^d} K(x,y) f(y)dy,\quad x\not\in {\rm supp} f, $$
where $K(x,y)$ satisfies the size and smoothness estimates
\begin{align}
&|K(x,y)|\leq {C\over |x-y|^d};\\
&|K(x + h, y) - K(x, y)| + |K(x, y + h)-K(x, y)| \leq C{ |h|^\delta\over |x-y|^{n+\delta}}\label{smooth}
\end{align}
for all $|x-y| > 2 |h| > 0$ and a fixed $\delta\in(0, 1]$.

\begin{theorem}\label{t:Compact suff} Suppose $ T$ is a Calder\'on--Zygmund operator as above,  $ 1< p < \infty $, 
two weights $  \lambda  , \sigma  \in A_p$ and $ \nu = (\sigma /\lambda  ) ^{1/p}$,  and $ b\in VMO _{\nu }(\mathbb R^d)$.  Then the commutator  $ [b,T] $ 
is compact from $ L ^{p} (\sigma ) $ to $ L ^{p} (\lambda  )$. 

\end{theorem}

The main idea of proving this argument is to  split $ [b,T] $ into two parts,  $ A$ and $B$, where the norm of $ A$ is at most $ \epsilon $, and $ B$ is a compact 
operator. 


Next we provide the argument for the necessity  of Theorem \ref{t:Main}. We note that this necessity argument holds for general operators with the non-degenerate condition on the kernel (formulated in \cite{H}). We state it as follows.

We say that the operator 
$T$ satisfies the \emph{non-degenerate condition}  if:
 {\it there exist positive constants $c_0$ and $ C_0$ such that for every $x\in \mathbb R^d$ and $r>0$, there exists $y\in B(x,  C_0 r)\backslash B(x,r)$  for which the kernel $ K (x,y)$ satisfies 
\begin{align}\label{e-assump cz ker low bdd weak}
|K(x, y)|\geq \frac1{c_0r^d}.
\end{align}
}

\begin{theorem}\label{t:Compact nece} Suppose $ 1< p < \infty $, two weights $  \lambda  , \sigma  \in A_p$ and $ \nu = (\sigma /\lambda  ) ^{1/p}$, $ b\in BMO _{\nu }(\mathbb R^d)$. Suppose that
$ T$  
satisfies the non-degenerate condition and that $[b,T]$ is compact from $ L ^{p} (\sigma ) $ to $ L ^{p} (\lambda  )$. Then $ b\in VMO _{\nu }(\mathbb R^d)$. 

\end{theorem}

The main idea of proving this argument  is to  seek a contradiction, which,  in its simplest form, is that there is no bounded operator $T \;:\; \ell^{p} (\mathbb N) \to \ell^{p} (\mathbb N)$ with $Te_{j } = T e_{k} \neq 0$ for all $j,k\in \mathbb N$.  Here, $e_{j}$ is the 
standard basis for $\ell^{p} (\mathbb N)$.  
Thus,
the main step is to construct  norm one, disjointly supported functions  $\{g_{j}\} \subset L ^{p} (\sigma )$ such that 
$  [b, T] g_{j} \approx  \phi  \neq 0 $.  

\begin{remark}
We point out that if $T$ is a Calder\'on--Zygmund operator  on $ L ^2 (\mathbb R^d )$ and $T$ satisfies the non-degenerate condition, then we obtain that there exist absolute  constants $3\le A_1\le A_2$ such that for any ball $B=B(x_0,r)$, there is another ball $\widetilde B:=B(y_0, r)$
such that $A_1 r\le |x- y|\le A_2 r $, 
and  for all $(x,y)\in ( B\times \widetilde{B})$, $K(x, y)$ does not change sign and \eqref{e-assump cz ker low bdd weak} holds. In fact, this argument is enough for us to deduce the above theorem.
\end{remark}

\begin{remark}
We also point out that our result and proof hold in a more general setting: spaces of homogeneous type, to cover many examples of Calder\'on--Zygmund operators beyond the Euclidean setting.
\end{remark}

We now study the weighted VMO space $VMO_\nu(\mathbb R^d)$, $\nu\in A_2$, which is of independent interest with the unweighted case known around 40 years ago.
We show that $VMO_\nu(\mathbb R^d)$ has totally different properties between $d=1$ and $d>1$.
\begin{theorem}\label{t:VMO}  
$VMO_\nu(\mathbb R)$ is the closure of $C_0^\infty(\mathbb R)$ under the $BMO_\nu(\mathbb R)$ norm.  
However, this is not necessarily true when $d\geq2$.
\end{theorem}

This paper is organized as follows. In Section \ref{s:pre} we provide the definitions for Muckenhoupt weights, weighted BMO and VMO spaces, and recall the representation theorem \cite{H}. In Section \ref{s:proof} we prove Theorem \ref{t:Main} via showing Theorems \ref{t:Compact suff} and \ref{t:Compact nece}. In Section \ref{s:VMO} we prove Theorem~\ref{t:VMO}. In the last section, we discuss the two weight compactness in the product setting with respect to little VMO spaces.

Throughout this paper, we use the standard notation “$A\lesssim B$” to denote $A \leq C B$ for some positive constant $C$ that depends only on the dimension $d$.


\section{Preliminaries} \label{s:pre}

\begin{definition}
  \label{def:Ap}
  Let $w(x)$ be a nonnegative locally integrable function
  on~$\mathbb R^d$. For $1 < p < \infty$, we
  say $w$ is an $A_p$-\emph{weight}, written $w\in
  A_p(\mathbb R^d)$, if
  \[
    [w]_{A_p}
    := \sup_B \left(\intav_B w\right)
    \left(\intav_B
      \left(\dfrac{1}{w}\right)^{1/(p-1)}\right)^{p-1}
    < \infty.
  \]

\end{definition}

Along the way, we will need different properties of $ A_p$ weights, which we will mention as they are needed.

Next we use $A_p$, $1 <p< \infty$, to denote the Muckenhoupt weighted class on $\mathbb R^d$ (see the precise definition of $A_p$ in Section 2), and the Muckenhoupt--Wheeden weighted BMO on $\mathbb R^d$ is defined as
follows.
\begin{definition}\label{MWbmo}
Suppose $w \in A_\infty$.
A function $b\in L^1_{\rm loc}(\mathbb R^d)$ belongs to
${ BMO}_w(\mathbb R^d)$ if
\begin{equation*}
\|b\|_{{ BMO}_w(\mathbb R^d)}:=\sup_{B}{1\over w(B)}\dint_{B}
|b(x)-b_{B}| \,dx<\infty,
\end{equation*}
where $b_{B}:= {1\over |B|}\int_B b(x)dx$ and the supremum is taken over all balls $B\subset \mathbb R^d$.
\end{definition}
The weighted VMO space on $\mathbb R^d$ is defined as
follows.
\begin{definition}\label{vmo}
Suppose $w \in A_\infty$.
A function $b\in { BMO}_w(\mathbb R^d)$ belongs to
${ VMO}_w(\mathbb R^d)$ if
\begin{align*}
&{\rm (i)}\ \lim_{a\to0}\sup_{B:\  r_B=a}{1\over w(B)}\dint_{B}
|b(x)-b_{B}| \,dx=0,\\
&{\rm (ii)}\ \lim_{a\to\infty}\sup_{B:\  r_B=a}{1\over w(B)}\dint_{B}
|b(x)-b_{B}| \,dx=0,\\
&{\rm (iii)}\ \lim_{a\to\infty}\sup_{B\subset \mathbb R^d\backslash B(x_0,a)}{1\over w(B)}\dint_{B}
|b(x)-b_{B}| \,dx=0,
\end{align*}
where $x_0$ is any fixed point in $\mathbb R^d$.
\end{definition}



\section{Proof of Main result: Theorem \ref{t:Main}} \label{s:proof}

It is clear that Theorem \ref{t:Main} follows from Theorem~\ref{t:Compact suff} and Theorem~\ref{t:Compact nece}.
In what follows, we provide the proofs of these two theorems.

\begin{proof}[Proof of Theorem~\ref{t:Compact suff}] 

This is seen as follows.  Suppose $ b \in VMO _{\nu }(\mathbb R^d)$ with $\|b\|_{BMO _{\nu }(\mathbb R^d)}=1$. 
We show that for any fixed $ 0< \epsilon < 1$, we have   $ [b,T] =  A+B$, where the norm of $ A$ is at most $ \epsilon $, and $ B$ is a compact 
operator. 

Fix $\eta>0$ small enough. Denote the kernel of $ T$ by $ K (x,y)$. Set $ K = \sum_{t=0} ^{3} K_t$, where each $ K_j$ is a Calder\'on--Zygmund kernel, and 
\begin{gather*}
 K_0 (x,y)  = 
 \begin{cases}
 K (x,y),  & 0<\lvert  x-y\rvert  < \eta, 
 \\
 0,   &  \lvert  x-y\rvert > 2\eta,  
 \end{cases}
\\
 K_1 (x,y)  = 
 \begin{cases}
 K (x,y),  & \lvert  x-y\rvert  > { 1/\eta}, 
 \\
 0,   &  \lvert  x-y\rvert < { 1/(2\eta)}
 \end{cases}
 \\
 K_2 (x,y) \neq 0 \quad \implies \quad  \lvert  x\rvert > 2/\eta {\rm\ \  or}\ \ \lvert  y\rvert > 2/\eta,  
 \\
 K _{3} (x,y)  \quad \textup{is supported on $ \lvert x \rvert, \lvert  y\rvert  < 2 / \eta  $} . 
\end{gather*}
Write $ T_j$ for the operator associated to the kernel $ K_j$.  We claim that 
\begin{equation*}
\lVert [b,T_j] \rVert _{L ^{p} (\sigma ) \to L ^{p} (\omega )} \lesssim \epsilon _ \eta , \qquad j=0,1,2,  
\end{equation*}
where $ \epsilon _ \eta $ decreases to zero as $ \eta $ does.  
Consider the case of $ j=0$. The   contribution to the norm estimate above  from $ BMO _{\nu }$ norm only arises from the 
weighted oscillation over the   cubes of side length at most $ \eta $, a fact that follows from the  the proof of the upper bound for the commutators in \cite{HLW} 
(or any of the other proofs for the upper bound).   For $ K_1$, the only contribution from $ b$ is the oscillation over cubes of side length at least $ 1/\eta $, and for $ K_2$, only oscillations for cubes which are either large, or at least a distance $ 1/ \eta $ from the origin.  In each case, the norm is small.  
We now choose $\eta=\eta(\epsilon)$ small enough, then we have $\lVert [b,T_j] \rVert _{L ^{p} (\sigma ) \to L ^{p} (\omega )} \lesssim \epsilon $ for $ j=0,1,2$.

It remains for us to argue that  $  [b,T_3]   $ is a compact operator.  
This is not quite trivial, due to the commutator structure, and the delicate nature of the weighted $ BMO$ space.  
For cubes $ P, Q$ of the same side length, let  $ K _{P,Q} (x,y)$ be a smooth kernel supported on $ P \times Q$, 
and with $ \lvert  K _{P,Q} (x,y)\rvert \lesssim \lvert  P\rvert ^{-2} $.  It follows from elementary facts about $ BMO _{ \nu }$ 
that the commutator 
\begin{equation*}
C _{P,Q} f (x) =  b (x) \int K_{P,Q} (x,y) f (y)\; dy -  \int K_{P,Q} (x,y) b (y)f (y)\; dy 
\end{equation*}
is bounded from $ L ^{p} (\sigma ) \to L ^{p} (\lambda )$, with norm that depends only on the relative positions of $ P$ and $ Q$.  
And, clearly, $ C _{P,Q} $ has compact range.  
The operator $ T_3$ can be well approximated by a finite sum $ \sum_{j} C _{P_j,Q_j} $.  Hence, $ [b,T_3] $ is compact. 
The proof of Theorem~\ref{t:Compact suff} is complete.
\end{proof}

\begin{proof}[Proof of Theorem~\ref{t:Compact nece}] 
Now assume that $b\in { BMO}_\nu(\mathbb R^d)$ such that $[b,T]$ is compact from $L^p(\sigma)$ to $L^p(\omega)$. 
But, for the sake of contradiction, further assume that   $b \not\in { VMO}_\nu(\mathbb R^d)$.  

The main idea
of getting contradiction is as follows:  on a Hilbert space $\mathcal H$, with canonical basis $e_j$, $j\in \mathbb N$, 
an operator $T $ with $T e_j = v$, with non-zero $v\in \mathcal H$, is necessarily unbounded.  We will see that, for example when $p=2$, a compact commutator from $L^2(\sigma)\to L^2(\omega)$
with symbol $b\in BMO_\nu(\mathbb R^d) \setminus VMO_\nu(\mathbb R^d)$ satisfies a variant of this condition.

Suppose that $b\not\in {\rm VMO}_\nu(\mathbb R^d)$, then at least one of the three conditions in Definition \ref{vmo} does not hold.  
The argument is similar in all three cases, and we just present the case that the first condition in Definition~\ref{vmo} does not hold.  
Then there exist $\delta_0>0$ and a sequence of balls $\{B_j\}_{j=1}^\infty=\{B_j(x_j,r_j)\}_{j=1}^\infty \subset \mathbb R^d$ such that $r_j\to 0$ as $j\to\infty$ and that
\begin{equation}\label{(i) not hold}
{1\over \nu(B_j)}\int_{B_j} |b(x)-b_{B_j}|dx\geq\delta_0.
\end{equation}
Without loss of generality, we can further assume that 
\begin{align}\label{ratio1}
4 r_{j_{i+1}}\leq  r_{j_{i}}.
\end{align}

According to the non-degenerate condition  \eqref{e-assump cz ker low bdd weak}: 
there exist  constants $3\le A_1\le A_2$ and a ball $\widetilde B_j:=B(y_j, r_j)$
such that $A_1 r_j\le |x_j- y_j|\le A_2 r_j$, 
and  for all $(x,y)\in ( B_j\times \widetilde{B}_j)$, $K(x, y)$ does not change sign and
\begin{equation}\label{e-assump cz ker low bdd}
|K(x, y)|\gs \frac1{r_j^d}.
\end{equation}

Let $m_{b}(\widetilde B_j)$ be a median value of $b$ on the ball $\widetilde B_j $. Namely,  
$m_{b}(\widetilde B_j)$ is a real number so that the two sets below have measure at least $ \frac{1}2 \lvert  \tilde B_j\rvert $.  
\begin{align*}
\nonumber F_{j,1}\subset\{ y\in \tilde B_j: b(y)\leq m_b(\tilde B_j) \},\qquad F_{j,2}\subset\{ y\in \tilde B_j: b(y)\geq m_b(\tilde B_j) \}.
\end{align*}
Next we define
$E_{j,1}=\{ x\in B: b(x)\geq m_b(\tilde B_j) \},\qquad E_{j,2}=\{ x\in B: b(x)< m_b(\tilde B_j) \}.$\ 
Then $B_j=E_{j,1}\cup E_{j,2}$ and $E_{j,1}\cap E_{j,2}=\emptyset$. 
And it is clear that 
\begin{equation}\begin{split}\label{bx-by0-1 1}
b(x)-b(y) &\geq 0, \quad (x,y)\in E_{j,1}\times F_{j,1},\qquad b(x)-b(y) &< 0, \quad (x,y)\in E_{j,2}\times F_{j,2}.
\end{split}\end{equation}
And for $(x,y)$ in $(E_{j,1}\times F_{j,1} )\cup (E_{j,2}\times F_{j,2})$, we have 
\begin{align}\label{bx-by-1 1}
|b(x)-b(y)| 
&=|b(x)-m_b(\tilde B_j)| +|m_b(\tilde B_j) -b(y)| \geq |b(x)-m_b(\tilde B_j)|. 
\end{align}

We now consider 
$$\widetilde F_{j,1}:= F_{j,1}\backslash \bigcup_{\ell=j+1}^\infty \tilde B_\ell\quad{\rm and}\quad \widetilde F_{j,2}:= F_{j,2}\backslash \bigcup_{\ell=j+1}^\infty \tilde B_\ell,\quad {\rm for}\ j=1,2,\ldots.$$
Then, based on the decay condition of the measures of  $\{B_j\}$ as in \eqref{ratio1} we obtain that
for each $j$,
\begin{align}\label{Fj1}
 |\widetilde F_{j,1}| &\geq |F_{j,1}|- \Big| \bigcup_{\ell=j+1}^\infty \tilde B_\ell\Big| \geq
{1\over 2} |\tilde B_j|-\sum_{\ell=j+1}^\infty \big|  \tilde B_\ell\big|\geq {1\over 2} |\tilde B_j|- {1\over 3}|\tilde B_j|= {1\over 6} |\tilde B_j|.
\end{align}
Similar estimate holds for $\widetilde F_{j,2}$.

Now for each $j$, we have that
\begin{align*}
&{1\over \nu(B_j)} \int_{B_j} |b(x)-b_{B_j}|dx\\
&\leq{2\over \nu(B_j)}\int_{B_{j}}\big|b(x)-m_b(\tilde B_j)\big|dx\\
&= {2\over \nu(B_j)}\int_{E_{j,1}}\big|b(x)-m_b(\tilde B_j)\big|dx + {2\over \nu(B_j)}\int_{E_{j,2}}\big|b(x)-m_b(\tilde B_j)\big|dx.
\end{align*}
Thus, combining with \eqref{(i) not hold} and the above inequalities, we obtain that as least one of the following inequalities holds:
\begin{align*}
{2\over \nu(B_j)}\int_{E_{j,1}}\big|b(x)-m_b(\tilde B_j)\big|dx \geq {\delta_0\over2},\quad 
{2\over \nu(B_j)}\int_{E_{j,2}}\big|b(x)-m_b(\tilde B_j)\big|dx \geq {\delta_0\over2}.
\end{align*}

Without lost of generality, we now assume that the first one holds, i.e., 
\begin{align*}
{2\over \nu(B_j)}\int_{E_{j,1}}\big|b(x)-m_b(\tilde B_j)\big|dx \geq {\delta_0\over2}.
\end{align*}

Therefore, for each $j$, from \eqref{e-assump cz ker low bdd} and \eqref{Fj1}  we obtain that 
\begin{align*}
{\delta_0\over4}&\leq{1\over \nu(B_j)}\int_{E_{j,1}}\big|b(x)-m_b(\tilde B_j)\big|dx\\
&\lesssim 
{1\over \nu(B_j)}{{|\widetilde F_{j,1}|}\over |B_j|}\int_{E_{j,1}}\big|b(x)-m_b(\tilde B_j)\big|dx\\
&\lesssim
{1\over \nu(B_j)}\int_{E_{j,1}}\int_{\widetilde F_{j,1}} |K(x,y)|\big|b(x)-b(y)\big|dydx.
\end{align*}
Next, since for $x\in E_{j,1}$ and $y\in \widetilde F_{j,1}$, $K(x,y)$ does not change sign and  
$b(x)-b(y)$ does not change sign either, we obtain that
\begin{align*}
{\delta_0}
&\lesssim {1\over \nu(B_j)}\bigg|\int_{E_{j,1}}\int_{\widetilde F_{j,1}} K(x,y)\big(b(x)-b(y)\big)dy\bigg|dx\\
&\lesssim
{1\over \sigma(B_j)^{1\over p}   \lambda'(B_j)^{1\over p'}}  \int_{E_{j,1}}\left|[b, T](\chi_{\widetilde F_{j,1}})(x)\right|dx\\
&=
{1\over    \lambda'(B_j)^{1\over p'}}  \int_{E_{j,1}}\left|[b, T]\bigg({\chi_{\widetilde F_{j,1}} \over \sigma(B_j)^{1\over p}}\bigg)(x)\right|dx,
\end{align*}
where $\lambda'(x) = \lambda^{-{1\over p-1}}(x)$, and  in the last equality, we use $p'$ to denote the conjugate index of $p$.

Next, by using H\"older's
inequality we further have
\begin{align*}
\delta_0 &\lesssim {1\over    \lambda'(B_j)^{1\over p'}}  \int_{E_{j,1}}\left|[b, T](f_j)(x)\right|  \lambda^{1\over p}(x) \lambda^{-{1\over p}}(x) dx\\
&\lesssim
 {1\over \lambda'(B_j)^{1\over p'}} \lambda'(E_{j,1})^{1\over p'} \bigg( \int_{\mathbb R^d}\big|[b, T](f_j)(x)\big|^p \lambda(x) dx\bigg)^{1\over p}\\
&\lesssim
\bigg( \int_{\mathbb R^d}\big|[b, T](f_j)(x)\big|^p \lambda(x)dx \bigg)^{1\over p},
\end{align*}
where in the above inequalities 
we denote
$$ f_j := {\chi_{\widetilde F_{j,1}} \over \sigma(B_j)^{1\over p}}.$$
This is a sequence of disjointly supported functions, by \eqref{Fj1}, with $ \|f_j\|_{L^p(\sigma)} \simeq  1 $.

Return to the assumption of compactness, and let $ \phi $ be in the closure of $  \{[b, T](f_j)\}_j$.  
We have $ \lVert \phi \rVert _{L ^{p} (\lambda )} \gtrsim 1$.   And, choose $ j_i$ so that 
$$ \|\phi- [b, T](f_{j_i})\|_{L^p(\lambda)} \leq  2^{-i}. $$

We then take non-negative numerical sequence $\{a_i\}$ with 
$$ \|\{a_i\}\|_{\ell^{p'}}<\infty \quad {\rm but}\quad   \|\{a_i\}\|_{\ell^1}=\infty.  $$  
Then, $ \psi =  \sum_{i} a_i f _{j_i} \in L ^{p} (\sigma ) $, and 
\begin{align*}
\Bigl\lVert  \sum_i {a_i} \phi   -  [ b, T ]{\psi } \Bigr\rVert_{L^p (\sigma )}  
&\leq \bigg\| \sum_{i=1}^\infty a_i   \bigl ( \phi - [b, T](f_{j_i}\bigr) \bigg\|_{L^p(\lambda)}
\\ & \leq \lVert  a_i\rVert _{\ell ^{p'}}  
\Bigl [   \sum_{i} \|\phi- [b, T](f_{j_i}) \| _{L^p(\lambda)} ^{p}  \Bigr] ^{1/p}  \lesssim 1. 
\end{align*}
So $ \sum_i {a_i} \phi  \in L^p (\sigma )$.  
But $ \sum_i{a_i}\phi $ is infinite on a set of positive measure. This is a contradiction that completes the proof.  
\end{proof}


\section{Properties for $VMO _{\nu }(\mathbb R^d )$: proof of Theorem~\ref{t:VMO}}  \label{s:VMO}

In this section we prove Theorem \ref{t:VMO}. We split it into two subsections. In the first subsection
we show that smooth compactly supported functions are dense in $ VMO _{\nu }$ in dimension one.  
 In the second, we construct a counterexample: a nice function $f$, which is not even in $BMO_\nu(\mathbb R^d)$, $d\geq2$.

\subsection{$VMO _{\nu }(\mathbb R )$ is the closure of ${C ^{\infty } _{0}}(\mathbb R) $ under the $BMO _{\nu }(\mathbb R )$ norm}
We  provide the following characterization of $VMO _{\nu }(\mathbb R )$, which is parallel to the well-known result in the unweighted setting, however, it is new in this weighted setting.

\begin{theorem}\label{t:Smooth}  For $ \nu \in A_2 (\mathbb R )$,  we have 
$
\overline   {C ^{\infty } _{0}} ^{ BMO _{\nu }(\mathbb R )} = VMO _{\nu }(\mathbb R ) 
$. 
\end{theorem}

This elementary Lemma is needed.  
\begin{lemma}\label{l:A2} Let $ \nu \in A_2(\mathbb R^d)$.  Then, 
\begin{itemize}
\item[(1)] we have $ \nu (\mathbb R ^{d}) = \infty $;
\item[(2)] there is a $ 0< d' = d' _{[\nu ] _{A_2}} <d$ so that  for any $ T>1$, we have 
\begin{equation}\label{e:inf}
\inf _{\substack{Q \;:\; Q\subset [-T,T] ^{d}\\  \ell (Q)  \leq 1}}  \frac{\nu (Q)} { \ell (Q) ^{d + d'}} >0. 
 \end{equation}
 \end{itemize}

\end{lemma}

Standard examples show that the second result is optimal. 
Let $ Q_T = [-T,T] ^{d}$. We will systematically suppress the dependence of various constants on the $ A_2$ constants of the weights.

\begin{proof}
For both, we argue by contradiction.  

 Assume $ \nu (\mathbb R ^{d}) < \infty $.  
The $ A_2$ product is always at least one. 
So, for all $ k \in \mathbb N $ we have 
$   \nu ^{-1} ({Q _{2 ^{k}}}) \gtrsim 2 ^{2kd}$.  And, so by equidistribution of $ \nu ^{-1} $,  
\begin{equation*}
\nu ^{-1} (Q _{2 ^{(k+1)}} \setminus Q _{2 ^{k}}) \gtrsim 2 ^{2kd}.  
\end{equation*}
Now, $ v ^{-1}$ is an $ A_2$ weight, so by Muckenhoupt's theorem, the maximal function $ M (f) $ is bounded from $ L ^2 (\nu ^{-1}  )$ to $ L ^2 (\nu ^{-1} )$, where $M$ is the standard Hardy--Littlewood maximal function on $\mathbb R^d$.  Apply it to $ \mathbf 1_{Q_1}$, so see that 
\begin{equation*}
\lVert M( \nu  \mathbf 1_{Q_1}) \rVert_{L^2 (\nu ^{-1} )} ^2 \gtrsim  \sum_{k \in \mathbb N }  2 ^{-2kd} \nu ^{-1} (Q _{2 ^{k+1}} \setminus Q _{2 ^{k}})  = \infty . 
\end{equation*}
This is a contradiction.  So $  \nu (\mathbb R ^{d}) = \infty $.  

\smallskip 

Given $ \nu \in A_2$, we have $ \nu ^{-1} \in A_2$, so there is a $ p >1$ so that $ \nu ^{-1}$ is in the Reverse Holder class $ RH _{p}$. 
Choose $ d' $ so that $ d' p = d$.  
Again, argue by contraction.  Fix $ T$ so that the infimum in \eqref{e:inf} is zero.
Then, we can find a sequence of cubes $ Q_j \subset Q_T$ so that 
each $ Q_j$ contains a set $ E_j$ with $ \{E_j\}$ being pairwise disjoint, and $2 \lvert  E_j\rvert > \lvert  Q_j\rvert  $.  
Finally, by equidistribution of $ A_2$ weights, $  \nu (x) \leq \ell( Q_j) ^{d'}$ for $ x\in E_j$.  Then, it follows that 
\begin{equation*}
\nu ^{-p} (Q_T) \geq \sum_{j}  \int _{E_j}\nu ^{-p}(x) \;dx  \gtrsim \sum_{j} 1 = \infty . 
\end{equation*}
But, $ \nu ^{-p} $ must be locally in $ L ^{1}$, so we have a contradiction, which yields (2) holds.
\end{proof}

\begin{lemma}\label{l:C2} For $ \nu \in A_2 (\mathbb R )$   we have  $  C ^2 _0(\mathbb R ) \subset  VMO _{\nu }(\mathbb R )$. 

\end{lemma}

\begin{proof}
From Definition \ref{vmo}, for every $b\in C ^2 _0(\mathbb R )$, it suffices to check the three conditions. The first  of these  conditions
is that the contribution from oscillations on large scales tends to zero. This follows from compactness and  $ \nu (\mathbb R ) = \infty $.  

The second concerns  medium scales. Oscillations should be bounded, but that follows from $ b$ being bounded. 
Third, oscillations at small scales should vanish. 
For $ b\in C ^2 _0(\mathbb R )$, we can assume that $ \lVert  D b\rVert _{\infty } \leq 1$. We have 
\begin{equation*}
b (x) - b (y) =  D b (\xi_y) \cdot  (x-y ) + O ( \lvert  x- y \rvert ^2  ). 
\end{equation*}
That is, in the direction of $ \nabla b (y)$, the difference grows like $ |x-y|$, but is otherwise of small order.  
Then, it follows that for a  interval $ I$ with side length at most one, 
\begin{equation*}
\frac{1} {\lvert  I\rvert }\int _{I}\int _{I} \lvert  b (x) - b (y) \rvert \;dy\, dx \lesssim   \lvert  I\rvert  ^{2}. 
\end{equation*}
This integral has to be divided by $ \nu (I) \gtrsim \ell( I) ^{1+ d'}$, where $ 0< d'<1$.  
But then from \eqref{e:inf}, the conclusion follows.  
\end{proof}

\begin{lemma}\label{l:2Smooth} For $ v\in A_2 (\mathbb R )$ we have  
$
\overline {C ^{2 }_0} ^{BMO _{\nu }(\mathbb R )} = \overline {C ^{ \infty  }_0} ^{BMO _{\nu }(\mathbb R )}  
$. 

\end{lemma}

\begin{proof}
For the proof, we need only show that 
$ \overline {C ^{2 }_0} ^{BMO _{\nu }(\mathbb R )} \subset  \overline {C ^{ \infty  }_0}  ^{BMO _{\nu }(\mathbb R )} . $
Fix $ b \in C ^2 _0(\mathbb R )$, and a smooth non-negative  compactly supported kernel $ \psi $, with integral equal to one. 
Set $ \psi _{j} (x)  = 2 ^{j} \psi (x 2 ^{j})$, for  $ j\in \mathbb N $.   
Then, $ \psi _{j} \ast b$ converges to $ b$ in the $  BMO _{\nu }(\mathbb R )$ norm, as follows from the proof of Lemma~\ref{l:C2}.
It is also in $ C ^{ \infty  }_0(\mathbb R )$.  Hence, $ b\in  \overline {C ^{ \infty  }_0}  ^{BMO _{\nu }(\mathbb R )}$. 
\end{proof}


\begin{proof}[Proof of Theorem~\ref{t:Smooth}]  
By Lemma~\ref{l:2Smooth}, it suffices to verify that $ \overline {C ^{2 }_0} ^{BMO _{\nu }(\mathbb R)} = VMO _{\nu }(\mathbb R)$. 
But, the inclusion `$ \subset $' is the content of Lemma~\ref{l:C2}. So we should show that every function $ \beta\in VMO _{\nu }(\mathbb R)$ 
can be approximated by a function in $ C ^2 _0(\mathbb R)$.  

But using the notation $ \psi _j$ from the previous proof, we have $ \psi _{j} \ast b$ converges to $ b$ in $ BMO $ norm. 
And $ \psi \ast b \in \overline {C ^{2 }_0} ^{BMO _{\nu }(\mathbb R)} $ so the proof is complete. 
\end{proof}

\subsection{Nice functions are not necessarily  in $VMO _{\nu }(\mathbb R^d )$ when $d>1$}

We construct an example of a smooth function and weight $ \nu \in A_2$  such that it is not even in $BMO _{\nu }(\mathbb R^d )$ when $d>1$.

The point is that $ A_2$ weights can  vanish at a point, with the vanishing order allowed to be as large as $ d$ on $ \mathbb R ^{d}$. 
There is no such requirement for smooth functions, of course.  
Let $ d=2$. 
Take $w(x) := |x|^{2-\epsilon}$ for   $\epsilon\in(0,1)$. Then $w\in A_2$ and moreover, we get
\begin{align*}
{1\over |B_r|} \int_{B_r} w(x) dx &={1\over |B_r|} \int_{B_r} |x|^{2-\epsilon} dx \approx r^{2-\epsilon}.
\end{align*}
Take 
\begin{align}
g(x) :=\left\{
                \begin{array}{ll}
                   x_1 \cdot e^{ - {1\over 1- |x|^2}} ,\quad |x|<1;\\[7pt]
                  \quad \ \ 0,\qquad |x|\geq1.
                \end{array}
              \right.
\end{align}
Then it is easy to see that $g\in C_0^\infty(\mathbb R^2)$,  $ \int_{B_r} g(x)dx =0 $ for any ball $B_r$  centered at the origin, and ${1\over |B_r|} \int_{B_r} |g(x)|dx \approx r$ for $ 0< r<1$. 
(That is, $ g$ has a zero of order $ 1$ at the origin.)

The function $g(x)$ is not in $BMO _{w }(\mathbb R^2 )$, since 
\begin{align*}
{1\over w(B_r)} \int_{B_r} |g(x) -g_{B_r}| dx &={1\over w(B_r)} \int_{B_r} |g(x) | dx\approx {1\over r^{1-\epsilon}},
\end{align*}
and this goes to $\infty$ as $r\to 0^+$.


\section{Product Setting: weighted little vmo space and compactness}  \label{s:product}

We show that our main results and methods above can be applied to 
the product setting to characterize the two weight compactness for commutators with respect to little vmo spaces. Note that the characterization of the two weight boundedness of commutators with respect to little bmo spaces was obtained in \cite{HPW}.
For the sake of simplicity, we only consider the commutator of double Riesz transforms and a symbol $b$ in little bmo, that is
$[b, \mathcal R_j^{(1)}\mathcal R_k^{(2)}]$, where $\mathcal R_j^{(1)}$ is the $j$th Riesz transform on $\mathbb R^{n_1}$ and $\mathcal R_k^{(2)}$ is the $k$th Riesz transform on $\mathbb R^{n_2}$.

To begin with, we now recall the product $A_p(\mathbb R^{n_1}\times \mathbb R^{n_2})$ weights.
\begin{definition}
  \label{def:Ap product}
  Let $w(x_1,x_2)$ be a nonnegative locally integrable function
  on~$\mathbb R^{n_1}\times \mathbb R^{n_2}$. For $1 < p < \infty$, we
  say $w$ is a product $A_p$ \emph{weight}, written as $w\in
  A_p(\mathbb R^{n_1}\times \mathbb R^{n_2})$, if
  \[
    [w]_{A_p(\mathbb R^{n_1}\times \mathbb R^{n_2})}
    := \sup_R \left(\intav_R w\right)
    \left(\intav_R
      \left(\dfrac{1}{w}\right)^{1\over p-1}\right)^{p-1}
    < \infty.
  \]
  Here the supremum is taken over all rectangles~$R:= I_1\times I_2\subset \mathbb R^{n_1}\times \mathbb R^{n_2}$, where $I_i$ is a cube in $\mathbb R^{n_i}$ for $i=1,2$.
  The quantity $[w]_{A_p(\mathbb R^{n_1}\times \mathbb R^{n_2})}$ is called the \emph{$A_p$~constant
  of~$w$}.
\end{definition}

Next we recall the weighted little bmo and vmo spaces  on $\mathbb R^{n_1}\times \mathbb R^{n_2}$.
\begin{definition}
  \label{def:bmo}
 For $1 < p < \infty$ and $w\in A_p( \mathbb R^{n_1}\times \mathbb R^{n_2})$, the weighted little bmo space ${\rm bmo}_w(\mathbb R^{n_1}\times \mathbb R^{n_2})$ is the space of all locally integrable functions $b$ on $\mathbb R^{n_1}\times \mathbb R^{n_2}$ such that
 $$ \|b\|_{{\rm bmo}_w(\mathbb R^{n_1}\times \mathbb R^{n_2})}= \sup_R {1\over w(R)}\int_{R}|b(x_1,x_2) - b_R|dx_1dx_2<\infty, $$
where  the supremum is taken over all rectangles ~$R:= I_1\times I_2\subset \mathbb R^{n_1}\times \mathbb R^{n_2}$, where $I_i$ is a cube in $\mathbb R^{n_i}$ for $i=1,2$.
\end{definition}

\begin{definition}
  \label{def:vmo p}
 For $1 < p < \infty$ and $w\in A_p( \mathbb R^{n_1}\times \mathbb R^{n_2})$, the weighted little vmo space ${\rm bmo}_w(\mathbb R^{n_1}\times \mathbb R^{n_2})$ is the space of all  $b$ in ${\rm bmo}_w(\mathbb R^{n_1}\times \mathbb R^{n_2})$ such that
 \begin{align*}
&{\rm (i)}\ \lim_{a\to0} \sup_{R:\  {\rm diam }(R) =a} {1\over w(R)}\int_{R}|b(x_1,x_2) - b_R|dx_1dx_2=0,\\
&{\rm (ii)}\ \lim_{a\to\infty}\sup_{R:\  {\rm diam }(R) =a} {1\over w(R)}\int_{R}|b(x_1,x_2) - b_R|dx_1dx_2=0,\\
&{\rm (ii)}\ \lim_{a\to\infty}\sup_{R\subset \mathbb R^d\backslash B((x^{(1)}_0,x^{(2)}_0),a)}{1\over w(R)}\int_{R}|b(x_1,x_2) - b_R|dx_1dx_2 =0,
\end{align*}
where $(x^{(1)}_0,x^{(2)}_0)$ is any fixed point in $\mathbb R^{n_1}\times\mathbb R^{n_2}$.
where  the supremum is taken over all rectangles $R:= I_1\times I_2\subset \mathbb R^{n_1}\times \mathbb R^{n_2}$, where $I_i$ is a cube in $\mathbb R^{n_i}$ for $i=1,2$.
\end{definition}

\begin{theorem}\label{t:Main product}  Suppose $ 1< p < \infty $,
two weights $  \lambda  , \sigma  \in A_p( \mathbb R^{n_1}\times \mathbb R^{n_2})$ and $ \nu = (\sigma /\lambda  ) ^{1/p}$. Suppose that $ b\in bmo _{\nu } (\mathbb R^{n_1}\times \mathbb R^{n_2} )$, $\mathcal R_j^{(1)}$ is the $j$th Riesz transform on $\mathbb R^{n_1}$ and $\mathcal R_k^{(2)}$ is the $k$th Riesz transform on $\mathbb R^{n_2}$, $j=1,2,\ldots,n_1,$ $k=1,2,\ldots,n_2$. 
Then we obtain that $ b\in vmo _{\nu } (\mathbb R^{n_1}\times \mathbb R^{n_2} )$ if and only if the commutator
\ $[b, \mathcal R_j^{(1)}\mathcal R_k^{(2)}]:  L ^{p} (\sigma )  \to  L ^{p} (\lambda  ) $\  is compact.
\end{theorem}

We point out that when $ b\in vmo _{\nu } (\mathbb R^{n_1}\times \mathbb R^{n_2} )$, the compactness argument 
follows from the same idea and technique as in the proof of Theorem \ref{t:Compact suff} with the splitting of the Riesz transforms $\mathcal R_j^{(1)}$ and $\mathcal R_k^{(2)}$ into four parts respectively. The reverse argument follows directly from the approach in the proof of Theorem \ref{t:Compact nece}. For the details, we omit here.

\end{document}